\documentclass[a4paper,11pt]{amsart} 
\usepackage{cite}
\usepackage{etoolbox}

\makeatletter
\patchcmd{\@setaddresses}{\indent}{\noindent}{}{}
\patchcmd{\@setaddresses}{\indent}{\noindent}{}{}
\patchcmd{\@setaddresses}{\indent}{\noindent}{}{}
\patchcmd{\@setaddresses}{\indent}{\noindent}{}{}
\makeatother

\usepackage[english]{babel}
\usepackage{dsfont} 
\usepackage{graphicx}
\usepackage{color}
\usepackage{enumerate}
\usepackage{array,multirow}

\allowdisplaybreaks

\usepackage{amssymb} 
\usepackage{amsthm} 
\usepackage{dsfont}
\usepackage{amsmath,esint}

\numberwithin{equation}{section}

\theoremstyle{plain}
\newtheorem{theorem}{Theorem}[section]
\newtheorem{lemma}[theorem]{Lemma}
\newtheorem{prop}[theorem]{Proposition}
\newtheorem{corollary}[theorem]{Corollary}
\newtheorem{remark}[theorem]{Remark}

\theoremstyle{definition}
\newtheorem{df}[theorem]{Definition}
\newtheorem{example}[theorem]{Example}

\begin{document}

\title{The moment problem on compact sets of characters 
}

\author{Dragu Atanasiu}

\address{Faculty of Textiles, Engineering and Business\newline University of Bor\aa s\newline All\'egatan 1, 503 32, Bor\aa s\newline  Sweden}

\email{dragu.atanasiu@hb.se}


\maketitle 
\begin{abstract}
In this paper we consider linear functionals on an unital commutative $\mathbb R$-algebra.
Recently has been  established a characterization for moment functionals, on compact sets of characters, depending only on the given functionals.

In this paper, we establish alternative characterizations.
For example, we obtain a characterization of a moment functional on a product of symmetric intervals  in which we do not assume that the functional is positive semidefinite but positive on an Archimedean semiring, and a characterization of a moment functional for which the compact support of the representing measure is a subset of the support of the representing measure in the characterization  mentioned at the beginning of the abstract.

We also establish a characterization of a moment functional which is positive on a Archimedean cone.

In connection  with the above characterization  we prove a Positivstellensatz for an Archimedean cone, which is neither a quadratic module nor a semiring.

At the end of the paper we consider an unital commutative semigroup $S$ with involution and characterize moment functions on compact sets of characters of $S$.
\end{abstract}

\section{Introduction}

Let $A$ be an unital commutative $\mathbb R$-algebra.
Let $X(A)$ be the set of all $\mathbb R$ algebras homomorphisms from $A$ to $\mathbb R$.
We assume that $X(A)$ is non-empty, and we endow $X(A)$ with the topology of pointwise convergence.
A linear functional $L:A\to \mathbb R$ 
is positive semidefinite if $L(a^2)\ge 0$ for all $a\in A$.

Let $K$ be a compact of $X(A)$. We say that the functional $L:A\to \mathbb R$ is a moment function on $K$ if there is a positive Radon measure $\mu $ on $K$
such that
$$L(a)=\int_K\alpha(a)d\mu(\alpha),a\in A.$$
The measure $\mu$ is called the representing measure for $L$.

Recently, Infusino et al.  have shown  in \cite{MSTP} that the linear functional $ L:A\to \mathbb R $ is a moment function on a compact of characters if and only if $L$ is positive semidefinite and  $\sup_{n\in \mathbb N}\sqrt[2n] {L(a^ {2n})}<\infty$ for all $a\in A$. Moreover, in this case, $L$ is a moment functional on the product of symmetric intervals
 $$K_L=\{\alpha \in X(A):|\alpha(a)|\le C_a\text { for all } a\in A\}$$
where $C_a=\sup_{n\in \mathbb N}\sqrt[2n] {L(a^ {2n})}$.

In Section 2 of this paper, we obtain an extension of the main result from \cite{MSTP} as a consequence of results from \cite{DA1} and \cite{DA2}.

The set $Q\subseteq A$ is a quadratic module if $1\in Q,Q+Q\subseteq Q$ and $A^ 2Q\subseteq Q$.
A quadratic module $Q$ is Archimedean if for every $a\in A$ there is a number $M_a>0$ such that $M_a\pm a\in Q$.

If $ L:A\to \mathbb R$ is a linear functional such that $L(1)=1 $ we denote
$$Q_L:=\{a\in A:L(b^ 2a)\ge 0\,\,\text{for all}\,\,b\in A\}.$$

Using the numbers $C_a,a\in A$ , we obtain, in Section 2 of the paper, a characterization of a functional, which is a moment functional on $K_L$ and where we do not assume that $L$ is positive semidefinite.

The set $S\subseteq A$ is a semiring if $S+S\subseteq S,S\cdot S\subseteq S$ and $\lambda\in S $ for every $\lambda\in \mathbb R,\lambda\ge 0.$

 The set $C\subseteq A$ is a cone if $C+C\subseteq C$ and $\lambda\cdot C\in \subseteq C $ for every $\lambda\ge 0.$

In \cite{SCH3} has been proved the  Positivstellensatz for an Archimedean semiring using the Positivstellensatz for an Archimedean quadratic module.

At the end of Section 2, we establish an  characterization of a functional, which is a moment functional on $K_L$ and where we do not assume that $L$ is positive semidefinite but positive on an Archimedean cone, which is neither a quadratic module nor a semiring, and then we prove a Positivstellensatz for this  cone.

 In Section 3, we establish the main results of this paper which are Theorem \ref{MRA} and Theorem \ref{gend_aD_a}. In  Theorem \ref{MRA} we establish an intrinsic  characterization  of the moment functionals, $ L:A\to \mathbb R $, on a compact set of characters where the support of the representing measure is contained in the set 
 $$D_L=\{\alpha \in X(A):d_a\le\alpha(a)\le D_a\text { for all }a\in A\},$$
 where $d_a,D_a$ are real numbers such that $-C_a\le d_a\le D_a\le C_a$ for all $a\in A$, which means that $D_L$ is a subset of $K_L$.
Theorem \ref{gend_aD_a} shows that the characterization from \cite{MSTP} and the  characterization from Theorem \ref{MRA} are, as expected, equivalent.

In Section 4 of this paper, we present applications of the characterizations from the previous sections.
For example, we prove a Berg-Maserick type theorem  and we discuss the moment problem on the ball  and the moment problem on a generalized simplex.

We also give, at the end of section 4, a proof of Schm\"udgen theorem concerning the moment problem on compact semi-algebraic sets, which use besides the ingredient of real algebraic geometry used in  \cite{SCH1}, the numbers $C_a$ and Theorem \ref{MR}.

In the last section of the paper, we consider an unital commutative semigroup $S$ with involution and characterize a positive semidefinite function, which is a moment function on a compact set of characters of $S$.

As an application, we relax the requirement, from \cite {ATZ} and \cite {DA2}, for a function, which is a moment function on a disc in $\mathbb C$.

\section{Intrinsic characterizations of moment functionals on a product of symmetric intervals }
In order to establish the main result of this section, we need the following theorem. We begin with a definition.
\begin{df}Let $A$ be an unital commutative $\mathbb R$-algebra. The function $v:A\to [0,\infty)$ is an absolute value if $v(ab)\le v(a)v(b)$ for all $a,b\in A$ and $v(1)=1$.

The functional $ L:A\to \mathbb R$ is $v-$bounded if there is a number $C>0$ such that 
$|L(a)|\le Cv(a),a\in A$.
\end{df}

\begin{theorem}\label{PreMR}Let $A$ be an unital commutative $\mathbb R$-algebra, a set  $Q\subseteq A$ and $v:A\to [0,\infty)$ an absolute value.

For a functional $ L:A\to \mathbb R$, the following conditions are equivalent:
\begin{enumerate}
\item[(i)]
$L$ is linear positive semidefinite , $v$-bounded and $Q\subseteq Q_L$;
\item[(ii)]there is an unique positive Radon measure $\mu$ on $K$ such that

$$L(a)=\int_{K}\alpha (a)d\mu(\alpha),a\in A$$
where

$$K=\{\alpha \in X(A):|\alpha(a)|\le v(a),a\in A\,\, \text{and}\,\,\alpha(q)\ge 0,q\in Q\}.$$

\end{enumerate}
\end{theorem}
\begin{proof}
Because the function $a\mapsto L(ab^2)$ is positive semidefinite and $v-$bounded we have by \cite{BCR}, Proposition 1.12 that $|L(ab^2)|\le v(a)L(b^2),a,b\in A$. This yields 
$$L(a^2b^2)=|L(a^2b^2)|\le v(a^2)L(b^2)\le (v(a))^2L(b^2),a,b\in A.$$
Consequently $(v(a))^2-a^2\in Q_L$. Now, the theorem follows from \cite{DA2}, Theorem 1.4.
\end{proof}

\begin{theorem}\label{MR}
Let $ L:A\to \mathbb R$ be a linear positive semidefinite functional and $Q\subseteq A$. Then, there exists an unique representing measure $\nu_L$ for $L$ with compact support, subset of the closed set $\{\alpha \in X(A)|\alpha (q)\ge 0,q\in Q\}$ if and only if 
$C_a=\sup_{n\in \mathbb N}\sqrt[2n] {L(a^ {2n})}<\infty,a\in A$ and $Q\subseteq Q_L.$

Moreover, in this case, the support of the measure $\nu_L$ is a subset of the set

$$\{\alpha \in X(A):|\alpha(a)|\le C_a\text { for all } a\in A,\alpha (q)\ge 0,q\in Q\}.$$

\end{theorem}
\begin{proof}[First proof]
Note that , because according to \cite{MSTP}, Lemma 3.3,   the function $a\mapsto \sup_{n\in \mathbb N}\sqrt[2n] {L(a^ {2n})}=C_a$ is an absolute value, 
the result is a consequence of Theorem \ref{PreMR}.
\end{proof}
\begin{proof}[Second proof]
From  \cite{DA1}, Theorem 1.3, we obtain that there exists a unique representing measure $\nu_L$ for $L$ with compact support 
a subset of the compact set 
\begin{equation}\begin{split}K=\{\alpha :A\to \mathbb R|\alpha (1)=1,\alpha(ab)=\alpha(a)\alpha(b),\\|\alpha(a)|\le C_a,a\in A,\alpha (q)\ge 0,q\in Q\}\end{split}\end{equation}
If  $\alpha\in K$ the equalities  
$$L((a+b)cc)-L(acc)-L(bcc)=0,a,b,c\in A$$
and
$$L(xacc)-x L(acc)=0,x\in \mathbb R,a,c\in A$$

show that $\alpha$ is linear and consequently
$$K=\{\alpha \in X(A):|\alpha(a)|\le C_a\text { for all } a\in A,\alpha (q)\ge 0,q\in Q\}.$$

\end{proof}

\begin{example}Let $ L:A\to \mathbb R$ be a linear positive semidefinite functional such that $\sup_{n\in \mathbb N}\sqrt[2n] {L(a^ {2n})}<\infty,a\in A$. We suppose that for every $a,b \in A$ we have $L(ab^ 2)\ge 0$. Then, there exists a unique representing measure $\nu_L$, for $L$,with a compact support  subset of the set

$$\{\alpha \in X(A):\alpha(a)\le C_a ,\alpha (a)\ge 0\text { for all }a\in A\}.$$

\end{example}

Note that in the case of the previous example, we have $Q_L=A$.

It results from the integral representation in Theorem \ref{MR} that $Q_L$ is a semiring. We use it in the next result, which is an intrinsic characterization where we do not assume that $L$ is positive semidefinite.
\begin{theorem}\label{Chproducts}
Let  $ L:A\to \mathbb R$ be a linear functional with $L(1)=1$. Then, there exists a unique representing measure $\nu_L$ for L with
compact support if and only if 

$$C_a=\sup_{n\in \mathbb N}\sqrt[2n] {L(a^ {2n})}<\infty,a\in A.$$
and   
$$L(C_{j_1,a_1}\dots C_{j_n,a_n})\ge 0,a_1,\dots,a_n\in A,j_1,\dots,j_n\in \{1,2\},n\in \mathbb N,$$
where 
$$C_{1,a}=C_a-a\quad\quad\quad C_{2,a}=C_a+a.$$

Moreover if there is a family of positive numbers $(T_a)_{a\in A} $ such that
$$L(T_{j_1,a_1}\dots T_{j_n,a_n})\ge 0,a_1,\dots,a_n\in A,j_1,\dots,j_n\in \{1,2\},n\in \mathbb N;$$
where 
$$T_{1,a}=T_a-a\quad\quad\quad T_{2,a}=T_a+a$$
then $C_a=\sup_{n\in \mathbb N}\sqrt[2n] {L(a^ {2n})}<\infty,a\in A$ ,$L$ is positive semidefinite 
and we have $C_a\le T_a,a\in A.$
\end{theorem}
\begin{proof}
Because 
$C_{1,a}+C_{2,a}=2C_a$ and $C_{2,a}-C_{1,a}=2a$
we obtain our result from \cite{M1}, Theorem 2.1.

The result is also a consequence of \cite{SCH4}, Corollary 12.47.
\end{proof}
Actually we can improve the result from Theorem \ref{Chproducts} in the following way:
\begin{theorem}
Let   $A$ be generated by a set $G$ and let $ L:A\to \mathbb R$ be a linear functional with $L(1)=1$. 

If there is a family of positive numbers $(T_a)_{a\in A} $ such that
$$L(T_{j_1,a_1}\dots T_{j_n,a_n})\ge 0,a_1,\dots,a_n\in G,j_1,\dots,j_n\in \{1,2\},n\in \mathbb N,$$
where  $T_{1,a}$ and $T_{2,a}$ are as in Theorem \ref{Chproducts},
then, $L$ is positive semidefinite and $C_a=\sup_{n\in \mathbb N}\sqrt[2n] {L(a^ {2n})}<\infty,a\in A$.

\end{theorem} 
\begin{proof}The result is an immediate consequence of  \cite{M1}, Theorem 2.1 and  \cite{MSTP},Theorem 1.2 .

\end{proof}
Now, we give another intrinsic characterization where we do not suppose that $L$ is positive semidefinite.

\begin{prop}\label{ConeARC}
Let  $ L:A\to \mathbb R$ be a linear functional with $L(1)=1$. Then, there exists a unique representing measure $\nu_L$ for L with
compact support if and only if 

$$C_a=\sup_{n\in \mathbb N}\sqrt[2n] {L(a^ {2n})}<\infty,a\in A$$
and
$$L((C_a-a)^j (C_a+a)^k)\ge 0,a\in A,j,k\in \mathbb N$$

\end{prop}
\begin{proof}This result is a consequence of \cite{MSTP}, Theorem 1.2 if we prove that for a  number $M>0$, the equalities 
$$L((M-a)^j (M+a)^k)\ge 0,a\in A,j,k\in \mathbb N,$$ 
imply that $L(a^2)\ge 0,a\in A$. This results from \cite{SCH3}, Lemma 3.1. We give an elementary proof of this result for the reader's convenience.

We have 
\begin{equation*}\begin{split}\sum_{j,k=0}^n\left(\frac{j^2(2M)^2+k^2(2M)^2}{n(n-1)}-\frac{2jk(2M)^2}{n^2}\right)\\\cdot {n\choose j} {n\choose k}(M+a)^j(M-a)^{n-j}(M-a)^k(M+a)^{n-k}\end{split}\end{equation*}

\begin{equation*}\begin{split}=\sum_{j,k=0}^n\Biggl(\frac{j(j-1)(2M)^2}{n(n-1)}+\frac{k(k-1)(2M)^2}{n(n-1)}-\frac{2jk(2M)^2}{n^2}\\
+\frac{j(2M)^2}{n(n-1)}+\frac{k(2M)^2}{n(n-1)}\Biggr) \\
\cdot {n\choose j} {n\choose k}(M+a)^j(M-a)^{n-j}(M-a)^k(M+a)^{n-k}\end{split}\end{equation*}

\begin{equation*}\begin{split}=(2M)^{2n}(M+a-(M-a))^2+\frac{(2M)^{2n+1}}{n-1}(M+a)\\+\frac{(2M)^{2n+1}}{n-1}(M-a)\end{split}\end{equation*}
$$=(2M)^{2n}4a^2+\frac{(2M)^{2n+1}}{n-1}(M+a)+\frac{(2M)^{2n+1}}{n-1}(M-a).$$
From the previous calculation, it is clear that equalities 
$$L((M-a)^j (M+a)^k)\ge 0,a\in A,j,k\in \mathbb N,$$ 
imply that $L(a^2)\ge 0,a\in A.$
\end{proof}
The method of showing that a functional is positive semidefinite, from the previous proof, was first used in \cite{MS}.

Note that if $ L:A\to \mathbb R$ is a linear functional positive semidefinite such that $C_a=\sup_{n\in \mathbb N}\sqrt[2n] {L(a^ {2n})}<\infty,a\in A$, we can show that
$$L((C_a-a)^j (C_a+a)^k)\ge 0,a\in A,j,k\in \mathbb N$$
without using the integral representation.  We can use instead \cite{Marshall},Proposition 5.2.3 or \cite{SCH2}, Lemma 12.6.

In the next result, we use the proof of Proposition \ref{ConeARC} to establish a positivstellensatz for an Archimedean cone, which is not a semiring nor a quadratic module.
\begin{prop}If  $(T_a)_{a\in A} $ is a family of positive numbers and $C$ the unital cone generated by the set $S$ where $S$ is
\begin{equation*}\begin{split}\left\{(T_a-a)^j (T_a+a)^k,a\in A,j,k\in \mathbb N\right\}\\
\cup\left\{(T^2_b-b^2)(T_a-a)^j (T_a+a)^k)\ge 0,a,b\in A,j,k\in \mathbb N\right\}\end{split}\end{equation*}
If $\alpha(c)>0$ for every $\alpha\in K$ where 

$$K=\{\alpha \in X(A):|\alpha(a)|\le T_a \text { for all }a\in A\}$$
 then $c\in C$.
\end{prop}
\begin{proof}
Because $T_a+a$ and $T_a-a$  are elements of  $C$ the unital cone $C$ is archimedean.
If we suppose that $c\notin C$, then, according to \cite{SCH4},Proposition 12.14, there is a linear functional $L:A\to \mathbb R$ positive on $C$ such that $L(1)=1$ and $L(c)\le 0$.

Using the proof of  Proposition \ref{ConeARC} and \cite{DA2}, Theorem 1.4 we obtain a measure $\mu$ on $K$ such that
$$L(a)=\int_{K}\alpha (a)d\mu(\alpha),a\in A$$
Now from $L(1)=1$ and $\alpha (c)>0,\alpha\in K$ we get $L(c)>0$. This is not possible and hence $c\in C$.
\end{proof}

Note that in the previous proof, the character $\alpha \in K$ if and only if $\alpha $ is positive on the set S.
\begin{remark}In the same way can be obtained a new proof for \cite{SCH3}, Theorem 3.3.
\end{remark}
In \cite{MSTP}, Lemma 3.3 is shown that $a\mapsto C_a$ is a seminorm.  Consequently, the equality from the next result is expected.
\begin{prop}Let $ L:A\to \mathbb R$ be a linear positive semidefinite functional with $L(1)=1$. We have
$$(C_a)^ 2=C_{a^ 2}=\sup_{b\in A,L(b^ 2)\ne 0} \frac{L(a^ 2b^ 2)}{L(b^ 2)}$$
\end{prop}
\begin{proof}
We denote 
$$\sup_{b\in A,L(b^ 2)\ne 0} \frac{L(a^ 2b^ 2)}{L(b^ 2)}=R_a.$$
We have
$$R_aL(b^ 2)-L(a^ 2b^ 2)\ge 0$$
for all $b\in A$.

Consequently from \cite{DA2},Theorem 1.4 there is a positive Radon measure $\nu $ on the compact 
$$K=\{\alpha \in X(A)||\alpha(a)|\le \sqrt{R_a}, a\in A\}$$ 
such that
$$L(a)=\int_K\alpha(a)d\nu (\alpha).$$

This yields 
$$L(a^ {2n})=\int_K|\alpha(a)|^ {2n}d\nu (\alpha)\le R_a^ {n},n\in \mathbb N.$$
Hence 
$$C_a\le \sqrt{R_a}.$$
Now, because $a\mapsto C_a$ is a seminorm, we get, as in the proof of Theorem \ref{PreMR},
$$C_{a^2}L(b^ 2)-L(a^ 2b^ 2)\ge 0$$
for all $b\in A$, which yields  
$$C_{a^ 2}\ge R_a.$$
This finishes the proof.
\end{proof}
As a direct consequence of the previous proposition we have the next  intrinsic characterization:
\begin{theorem}
Let  $ L:A\to \mathbb R$ be a linear positive semidefinite functional with $L(1)=1$. Then, there exists a unique representing measure $\nu_L$ for L with
compact support if and only if 
$$\sup_{b\in A,L(b^ 2)\ne 0} \frac{L(a^ 2b^ 2)}{L(b^ 2)}<\infty,a\in A.$$
\end{theorem}
\section{An alternative intrinsic characterization of moment functionals in the compact case }
To establish the characterization of moment functionals in this section, we will need the following theorem
\begin{theorem}\label{chaintervals}Let $A$ be an unital commutative $\mathbb R$-algebra and $Q\subseteq A$. Let $(M_a)_{a\in A}$ and $(m_a)_{a\in A}$ families of  real numbers such that $m_a\le M_a$ and 
$$K=\{\alpha \in X(A)|m_a\le \alpha(a)\le M_a, a\in A\,\, \text{and}\,\,\alpha (q)\ge 0,q\in Q\}.$$
For a linear functional $ L:A\to \mathbb R$ the following conditions are equivalent:
\begin{enumerate}
\item[(i)]the functional $L$ is positive semidefinite, 
$a-m_a,M_a-a\in Q_L$ for all $a\in A$ and  $q\in Q_L$ for all  $q\in Q$ ;
\item[(ii)]there is a unique positive Radon measure $\mu$ on $K$ such that

$$L(a)=\int_{K}\alpha (a)d\mu(\alpha),a\in A.$$
\end{enumerate}
\end{theorem}

\begin{proof}
We have $|m_a|+|M_a|+a=a-m_a+m_a+|m_a|+|M_a|\in Q_L$ and $|m_a|+|M_a|-a=M_a-a-M_a+|m_a|+|M_a|\in Q_L$.

Let $M=|m_a|+|M_a|$. Because the functions  $x\mapsto L((M-a)x)$ and $x\mapsto L((M+a)x)$ are positive semidefinite the function $x\mapsto L((M^ 2-a^ 2)x)$ is positive semidefinite as a consequence of the equality
$$(M-a)(M+a)^ 2+(M+a)(M-a)^ 2=2M(M^ 2-a^ 2)$$

Hence, \cite{DA2}, theorem 1.4 yields the result from the theorem.
\end{proof}
\begin{remark}This theorem can be obtained also from \cite {DA3},Theorem 2.1.
A particular case when $0\le d_a\le D_a$ is in \cite{DA4},Theorem 3.
\end{remark}
\begin{theorem}\label{MRA}
Let $A$ be an unital commutative $\mathbb R$-algebra.
Let $L: A\to\mathbb R $ be a  positive semidefinite functional on $A$ with $L(1)=1$.
Then, there exists a unique representing Radon measure $\mu_L$ for $L$ with compact support  if and only if
$$\sup_{x\in A,L(x^ 2)\ne 0} \frac{L(ax^ 2)}{L(x^ 2)}<\infty,\,\,\text{for all }a\in A$$
and

$$\inf_{x\in A,L(x^ 2)\ne 0} \frac{L(ax^ 2)}{L(x^ 2)}>-\infty,\,\,\text{for all }a\in A.$$
Moreover, in this case $supp(\nu_L)$ is a subset of the set $K$  of all $\alpha\in X(A)$ such that 
$$\inf_{x\in A,L(x^ 2)\ne 0} \frac{L(ax^ 2)}{L(x^ 2)}\le \alpha(a)\le \sup_{x\in A,L(x^ 2)\ne 0} \frac{L(ax^ 2)}{L(x^ 2)}$$
 and we have,
$$S_a=\sup_{\alpha\in K}\alpha(a)=\sup_{x\in A,L(x^ 2)\ne 0} \frac{L(ax^ 2)}{L(x^ 2)}$$
and
$$s_a=\inf_{\alpha\in K}\alpha(a)=\inf_{x\in A,L(x^ 2)\ne 0} \frac{L(ax^ 2)}{L(x^ 2)}$$
for all $a\in A$.
\end{theorem}
\begin{proof}
We denote 
$$D_a=\sup_{x\in A,L(x^ 2)\ne 0} \frac{L(ax^ 2)}{L(x^ 2)}$$
and
$$d_a=\inf_{x\in A,L(x^ 2)\ne 0} \frac{L(ax^ 2)}{L(x^ 2)}$$
and suppose that  $d_a,D_a\in\mathbb R$.
 We have 

$D_aL(x^ 2)-L(ax^ 2)\ge 0,x\in A$
and
$L(ax^ 2)-d_aL(x^ 2)\ge 0,x\in A$.

As a consequence of Theorem  \ref{chaintervals} there is a positive Radon measure $\mu_L$ on $K=\{\alpha \in X(A)|d_a\le \alpha(a)\le D_a, a\in A\}$ such that

$$L(a)=\int_{K}\alpha (a)d\mu_L(\alpha),a\in A. $$
From the integral representation we get $S_a\le D_a$ and $d_a\le s_a$.

Also, using the integral representation  it is easy to see that 
$$S_aL(x^ 2)-L(ax^ 2)\ge 0,x\in A$$
and
$$L(ax^ 2)-s_aL(x^ 2)\ge 0,x\in A.$$
Consequently $S_a\ge D_a$ and $s_a\le d_a$.

Now if there is a positive Radon measure $\nu$ on a compact $K^\prime\subseteq X(A)$ such that
$$L(a)=\int_{K^\prime}\alpha (a)d\nu(\alpha),a\in A $$
we have, for all $a,x\in A$: 
$$L(ax^ 2)=\int_{K^\prime}\alpha (a)\alpha(x^2)d\nu(\alpha)\le \sup _{\alpha\in K^\prime} \alpha(a)\int_{K^\prime}\alpha(x^2)d\nu(\alpha)\le \sup _{\alpha\in K^\prime}\alpha(a)L(x^2)$$
which yields $D_a\in \mathbb R$ and in the same way we show that $d_a\in \mathbb R$.

This finishes the proof.
\end{proof}

In the following result we do not suppose that the functional is positive semidefinite.
\begin{theorem}\label{prodD}
Let  $ L:A\to \mathbb R$ be a linear functional with $L(1)=1$. Then, there exists a unique representing measure $\nu_L$ for L with
compact support if and only if 

$$D_a=\sup_{x\in A,L(x^ 2)\ne 0} \frac{L(ax^ 2)}{L(x^ 2)}<\infty\,\,,d_a=\inf_{x\in A,L(x^ 2)\ne 0} \frac{L(ax^ 2)}{L(x^ 2)}>-\infty$$
and   
$$L(D_{j_1,a_1}\dots D_{j_n,a_n})\ge 0,a_1,\dots,a_n\in A,j_1,\dots,j_n\in \{1,2\},n\in \mathbb N,$$
where 
$$D_{1,a}=D_a-a\quad\quad\quad D_{2,a}=a-d_a.$$
\end{theorem}
\begin{proof}
The result is a consequence of \cite{M1}, Theorem 2.1 or  \cite{SCH4}, Corollary 12.47.
\end{proof}
\begin{prop}\label{prelmain} Let $A$ be an unital commutative $\mathbb R$-algebra.
Let $L:A\to\mathbb R $ be a  positive semidefinite functional on $A$ with $L(1)=1$.
Then there exists a unique representing Radon measure $\nu_L$ for $L$ with compact support  if and only if the sets $\{M\in \mathbb R: M-a\in Q_L\}$ and $\{M\in \mathbb R: a-M\in Q_L\}$
are  non empty for every $a\in A$ and we have
$$D_a=\inf \{M\in \mathbb R: M-a\in Q_L\}$$
and
$$d_a=\sup \{M\in \mathbb R: a-M\in Q_L\}$$
\end{prop}
\begin{proof}
Suppose that  the sets $\{M\in \mathbb R: M-a\in Q_L\}$ and $\{M\in \mathbb R: a-M\in Q_L\}$ are non empty.

We denote 
$$\inf \{M\in [0,\infty)| M-a\in Q_L\}=D_a^\prime$$
and
$$\sup \{M\in \mathbb R: a-M\in Q_L\}=d_a^\prime$$
We have
$$D_a^\prime L(b^ 2)-L(ab^ 2)\ge 0$$
and
$$L(ab^ 2)-d_a^\prime L(b^ 2)\ge 0$$
for every $b\in A$.

Consequently $D_a^\prime\ge D_a$ and $d_a^\prime\le d_a$.

Now suppose that $$D_a=\sup_{x\in A,L(x^ 2)\ne 0} \frac{L(ax^ 2)}{L(x^ 2)}<\infty,\,\,\text{for all }a\in A$$
and
$$d_a=\inf_{x\in A,L(x^ 2)\ne 0} \frac{L(ax^ 2)}{L(x^ 2)}>-\infty,\,\,\text{for all }a\in A$$

Because
$$D_a L(b^ 2)-L(ab^ 2)\ge 0$$
and
$$L(ab^ 2)-d_a L(b^ 2)\ge 0$$
for every $b\in A$,  $D_a-a\in Q_L$ and $a-d_a\in Q_L$ and we get:
$D_a^\prime\le D_a$ and $d_a^\prime\ge d_a$.
\end{proof}
The next result shows that the characterization from \cite{MSTP} and the characterization from this section are equivalent.
\begin{theorem}\label{gend_aD_a} Let $A$ be an unital commutative $\mathbb R$-algebra.
Let $L:A\to\mathbb R $ be a  positive semidefinite linear functional on $A$ with $L(1)=1$.

Then  $C_a=\sup_{n\in \mathbb N}\sqrt[2n] {L(a^ {2n})}<\infty,a\in A$ if and only if  $D_a<\infty$ and $d_a>-\infty$.

Moreover, we have $-C_a\le d_a\le D_a\le C_a$ and at least one of the equalities $C_a=D_a$ or $-C_a=d_a$ is true.
\end{theorem}
\begin{proof}Suppose $C_a=\sup_{n\in \mathbb N}\sqrt[2n] {L(a^ {2n})}<\infty,a\in A$. Because the integral representation from Theorem \ref{MR} yields $C_a-a\in Q_L$ and $C_a+a\in Q_L$ it results from Proposition \ref{prelmain} that $C_a\ge D_a$ and $-C_a\le d_a$.

Now if $D_a<\infty$ and $d_a>-\infty$ we have, using theorem \ref{MRA}
$$L(a)=\int_{K}\alpha(a)d\nu (\alpha),a\in A$$
where
$$K=\left\{\alpha\in X(A):d_a\le \alpha(a)\le D_a\,\,\text{for all }a\in A\right\}$$
and $\nu$ is a positive Radon measure on $K$.
This  yields, as in \cite{MSTP}, proof of Theorem 1.2, $C_a=\sup_{n\in \mathbb N}\sqrt[2n] {L(a^ {2n})}<\infty,a\in A$.

Now there are four cases
\begin{description}
\item [1. $D_a\ge0$ and $ d_a\ge 0$] in this case  $C_a\le D_a$ and consequently $C_a= D_a$;

\item [2.  $D_a\le0$ and $ d_a\le 0$] in this case $C_a\le -d_a$ and consequently $-C_a= d_a$;

\item [3. $D_a\ge0$ and $ d_a\le 0$ and $D_a\ge -d_a$] in this case $C_a\le D_a$ and consequently $C_a= D_a$;

\item [4. $D_a\ge0$ and $ d_a\le 0$ and $D_a\le -d_a$] in this case $C_a\le -d_a$ and consequently $-C_a= d_a$.
\end{description}
\end{proof}
\begin{corollary}The function $a\mapsto \max(-d_a,D_a)$ is a seminorm on the algebra $A$.
\end{corollary}
\begin{proof}
This is a consequence of the equality $C_a=\max(-d_a,D_a)$ and of \cite {MSTP}, Lemma 3.3.
\end{proof}

\begin{prop}\label{d_aD_agen}
Let   $A$ be generated by a set $G$ and let $ L:A\to \mathbb R$ be a linear functional. If  $D_a<\infty$ and $d_a>-\infty$ for all $a\in G$ then  $D_a<\infty$ and $d_a>-\infty$ for all $a\in A$.

\end{prop} 
\begin{proof}If  $D_a<\infty$ and $d_a>-\infty$ for all $a\in G$ it results, as in the proof of Theorem \ref{gend_aD_a} that $C_a<\infty $ for all $a\in G$. Now from \cite{MSTP}, Proposition 4.3 
it results that $C_a<\infty $ for all $a\in A$. Consequently again as in the proof of Theorem \ref{gend_aD_a} we obtain $D_a<\infty$ and $d_a>-\infty$ for all $a\in A$.

\end{proof}
Using Proposition \ref{d_aD_agen} we relax, in the next two results, the requirements in the Theorems \ref{MRA} and \ref{prodD}.These results are interesting for applications to the algebra 
$A=\mathbb R[X_1,\dots,X_m]$.
\begin{theorem}\label{MRAG}
Let $A$ be an unital commutative $\mathbb R$-algebra generated by a set $G$.
Let $L: A\to\mathbb R $ be a  positive semidefinite functional on $A$.
Then, there exists a unique representing Radon measure $\nu_L$ for $L$ with compact support  if and only if
$$D_a<\infty,\,\,\text{for all }a\in G$$
and

$$d_a>-\infty,\,\,\text{for all }a\in G.$$
Moreover in this case the support of the representing measure for $L$ is a subset of the set

$$D_L=\{\alpha \in X(A):d_a\le\alpha(a)\le D_a\text { for all }a\in A\}$$
and we have
$$D_L=\{\alpha \in X(A):d_a\le\alpha(a)\le D_a\text { for all }a\in G\}.$$
\end{theorem}
\begin{proof}
The result is a consequence Theorem \ref{MRA} and Proposition \ref{d_aD_agen}.
\end{proof}

\begin{theorem}
Let   $A$ be generated by a set $G$ and let $ L:A\to \mathbb R$ be a linear functional. Then, there exists a unique representing measure $\nu_L$ for L with
compact support if and only if 

$$D_a<\infty\,\,,d_a>-\infty,a\in G$$
and   
$$L(D_{j_1,a_1}\dots D_{j_n,a_n})\ge 0,a_1,\dots,a_n\in G,j_1,\dots,j_n\in \{1,2\},n\in \mathbb N,$$
where $D_{1,a}$ and $D_{2,a}$ are as in Theorem \ref{prodD}.

Moreover in this case the support of the representing measure for $L$ is a subset of the set $D_L$ defined in Proposition \ref{MRAG}.

\end{theorem}
\begin{proof}
The result is a consequence of \cite{M1}, Theorem 2.1 and of  Theorem \ref{MRA} and Proposition \ref{d_aD_agen}.
\end{proof}
\section{Applications of the intrinsic characterizations}

\subsection{A  Berg-Maserick type result}
Inspired by \cite{SCH2}, Corollary 12.16, we get the following result, which is an immediate consequence of \cite{MSTP},Proposition 4.3 and Theorem \ref{MR}.
Note that in this theorem we relax the requirements from Theorem \ref{PreMR}. See also  \cite{BCR}, Theorem 2.5 and \cite{BM}, Theorem 2.1.
\begin{theorem}\label{genBM}Let $A$ be an unital commutative $\mathbb R$-algebra and $G$ a set of generators for $A$. Let  $(\gamma_a)_{a\in G}$ and $(\delta_a)_{a\in G}$  be families of positive real numbers.

For a function $ L:A\to \mathbb R$ the following conditions are equivalent:
\begin{enumerate}
\item[(i)] L is a linear positive semidefinite functional and
$$L(a^{2n})\le \gamma_a \delta_a^ {2n},n\in\mathbb N,a\in G;$$
\item[(ii)]There is a positive Radon measure $\mu$ on $K$ such that

$$L(a)=\int_{K}\alpha (a)d\mu(\alpha),a\in A$$
where

$$K=\{\alpha \in X(A):|\alpha(a)|\le \delta_a,a\in G\}.$$
\end{enumerate}
\end{theorem}

\subsection{The moment problem on a ball}

\begin{prop}\label{sf}
Let $A$ be an unital algebra with generators $a_1,\dots,a_m$.
Let $L:A\to\mathbb R $ be a linear positive semidefinite function on $A$ such that $L(1)=1$.
Then there exists a unique representing Radon measure $\nu_L$ for $L$ with compact support  if and only if 
$$C_{a_1^ {2}+\dots+a_m^ {2}}=\sup_{n\in \mathbb N}\sqrt[2n] {L((a_1^ {2}+\dots+a_m^ {2})^ {2n})}<\infty$$

Moreover, in this case 
$$supp(\nu_L)\subseteq\left\{\alpha\in X(A):|\alpha(a_1)|^ 2+\dots+|\alpha(a_m)|^ 2\le C_{a_1^ {2}+\dots+a_m^ {2}}\right\}.$$
\end{prop}
\begin{proof}
Because $L$ is positive semidefinite, we have 
$$\left(L(a_j^ {2n})\right)^ 2\le L((a_1^ {2}+\dots+a_m^ {2})^ {2n})$$
for all $j\in\{1,\dots,m\}$ and $n\in \mathbb  N$.
Now the result is an immediate consequence of Theorems \ref{genBM} and  \ref{MR} and \cite{MSTP},Proposition 4.3.
\end{proof}
Using the previous result, we can obtain the following solution of the moment problem on a ball.

\begin{prop}\label{Msf}
Let $A$ be an unital algebra with generators $a_1,\dots,a_m$.
For a linear function  $L:A\to\mathbb R $ such that $L(1)=1$ the following conditions are equivalent :
\begin{enumerate}

\item the functional $L$ is positive semidefinite, and we have  
$$C_{a_1^ {2}+\dots+a_m^ {2}}=\sup_{n\in \mathbb N}\sqrt[2n] {L((a_1^ {2}+\dots+a_m^ {2})^ {2n})}\le r^2;$$
\item the functional $L$ is positive semidefinite, and we have  
$$r^2L(b^ 2)-L(b^ 2(a_1^2+\dots+a_m^2))\ge 0\,\,\text{for all}\,\,b\in A;$$
\item there is a unique positive Radon measure $\mu$ on 
$$K=\left\{\alpha\in X(A):|\alpha(a_1)|^ 2+\dots+|\alpha(a_m)|^ 2\le r^2\right\}$$
such that
$$L(a)=\int_K\alpha(a)d\mu(\alpha), a\in A.$$
\end{enumerate}
\end{prop}
\begin{proof}The implication $1\implies 3 $ results from Proposition \ref{sf}. The equivalence $2\iff 3 $ is a consequence of \cite{DA2},Theorem  3.1. This finishes the proof because the implication $3\implies 1$ is immediate.

\end{proof}
\begin{remark}The equivalence $2\iff3 $ is proved in a different way in \cite{DA2}.
\end{remark}
\subsection{The moment problem on a generalized simplex}

\begin{prop}\label{sim}
Let $A$ be an unital algebra with generators $a_1,\dots,a_m$.
Let $L:A\to\mathbb R $ be a linear positive semidefinite function on $A$ such that $L(1)=1$.
Then there exists a unique representing Radon measure $\nu_L$ for $L$ with compact support  if and only if 
 $d_{a_1},\dots,d_{a_m}>\infty$ and $D_{a_1+\dots+a_m}<\infty$.

Moreover, in this case $supp(\nu_L)$ is a subset of the set
\begin{equation*}\begin{split}\Biggl\{\alpha\in X(A):\alpha(a_1)\ge d_{a_1},\dots\alpha(a_m)\ge d_{a_m},\\\alpha(a_1)+\dots+\alpha(a_m)\le D_{a_1+\dots+a_m}\Biggr\}\end{split}\end{equation*}
\end{prop}
\begin{proof}If the support of $\nu_L$ is compact we have
$a_1-d_1,\dots a_m-d_m\in Q_L$ and $D_1-a_1,\dots D_m-a_m\in Q_L$
which yields $D_1+\dots+D_m-(a_1+\dots+a_m)\in Q_L$ and consequently $D_{a_1+\dots+a_m}<\infty$.

Now if  $a_1-d_1,\dots a_m-d_m\in Q_L$ and $D_{a_1+\dots+a_m}-(a_1+\dots+a_m)\in Q_L$.
then for every $j\in\{1,\dots,m\}$ we have 
$$D_{a_1+\dots+a_m}+d_1+\dots+ d_{j-1}+d_{j+1}+\dots d_m-a_j\in Q_L$$
and it results from Theorem \ref{MRA} that  the support of $\nu_L$ is compact.

The rest of the proof is a consequence of Theorem \ref{MR}.
\end{proof}
\begin{example}\label{Msi}
In this example we consider the algebra $A=\mathbb R[X_1,\dots,X_m]$. The space of characters $X(A)$ can be identified with $\mathbb R^ m$. We have:

For a linear function  $L:\mathbb R[X_1,\dots,X_m]\to\mathbb R $ the following conditions are equivalent :
\begin{enumerate}

\item the functional $L$ is positive semidefinite, and we have  
 $d_{X_1}\ge 0,\dots,d_{X_m}\ge 0$ and $D_{X_1+\dots+X_m}\le 1$;
\item there is a unique positive Radon measure $\mu$ on 
$$K=\left\{(x_1,\dots,x_m)\in\mathbb R^ m:x_1\ge 0,\dots,x_m\ge 0,x_1+\dots+x_m\ge 1\right\}$$
such that
$$L(X_1^{ n_1}\dots X_m^{ n_m})=\int_Kx_1^{ n_1}\dots x_m^{ n_m}d\mu(x_1,\dots,x_m), (n_1,\dots,n_m)\in \mathbb N_0^ m.$$
\end{enumerate}
\end{example}

\subsection{A second Berg-Maserick type result}

\begin{df}The function $v:A\to [0,\infty)$ is a weak absolute value if $v(a^ 2)\le v(a)^ 2$ for all $a\in A$ and $v(1)=1$.
\end{df}
Now we relax as in \cite{GMW}, p.298 the requirement on the absolute value from Theorem \ref{PreMR}.
\begin{theorem}Let $A$ be an unital commutative $\mathbb R$-algebra, $Q\subset A$,  and $v:A\to [0,\infty)$ a weak absolute value.

For a functional $ L:A\to \mathbb R$ such that $L(1)=1$,
the following conditions are equivalent:
\begin{enumerate}
\item[(i)]
$L$ is linear positive semidefinite, $v$-bounded and $Q\subseteq Q_L$ ;
\item[(ii)]There is a positive Radon measure $\mu$ on $K$ such that

$$L(a)=\int_{K}\alpha (a)d\mu(\alpha),a\in A$$
where $K=\{\alpha \in X(A):|\alpha(a)|\le v(a),a\in A\,\, \text{and}\,\,\alpha(q)\ge 0,q\in Q\}$

\end{enumerate}
\end{theorem}
\begin{proof}
Let $a\in A$ and $d\in \mathbb N$. We have 
$$ L(a^ {2^{d+1}})\le Cv(a^ {2^{d+1}})\le C(v(a^ {2^{d}}))^ 2\le\dots\le C(v(a))^{2^{d+1}}.$$
Because the sequence $\{\sqrt[2n] {L(a^ {2n})}\}_{n\in \mathbb N}$ is monotone increasing we have 
$$\sup_{n\in \mathbb N}\sqrt[2n] {L(a^ {2n})}=\sup_{d\in \mathbb N}\sqrt[2^ d] {L(a^ {2^d})}$$
which yields  $C_a\le v(a)$. Now the result is an immediate consequence of Theorem \ref{MR}.
\end{proof}
\subsection{Characterization of the elements of $Q_L$}
\begin{prop}\label{MR2}
Let $ L:A\to \mathbb R$ be a linear positive semidefinite functional with L(1)=1 and $C_a=\sup_{n\in \mathbb N}\sqrt[2n] {L(a^ {2n})}<\infty,a \in A$. Then, $a\in Q_L$ if and only if   $C_{C_a-a}\le C_a$.

\end{prop}
\begin{proof}Let $$K_L=\{\alpha \in X(A)||\alpha(a)|\le C_a, a\in A\}.$$
If $a\in Q_L$ then $0\le\alpha(a)\le C_a$,$\alpha\in K_L$, and we have 
$$L(( C_a-a)^{2n})=\int_ {K_L}(C_a-\alpha(a) )^{2n} d\nu({\alpha})\le C_a^ {2n}$$
which yields
$$C_{C_a-a}\le C_a.$$
Now suppose that $C_{C_a-a}\le C_a$. Hence, if $\alpha\in K_L$, then
$$|\alpha(C_a-a)|=|C_a-\alpha(a)|=C_a-\alpha(a)\le C_{C_a-a}\le C_a$$
that is $\alpha(a)\ge 0$ for all  $\alpha\in K_L$ and consequently $a\in Q_L$.
\end{proof}
Note that the implication $a\in Q_L$ implies $C_{C_a-a}\le C_a$ is proved in a diferrent way in \cite{MSTP},Lemma 3.4.

\subsection{A Bochner-Godement type theorem}
\begin{lemma}\label{minC_a}Let $A$ be an unital commutative $\mathbb R$-algebra.
Let $L:A\to\mathbb R $ be a  positive semidefinite functional on $A$ with L(1)=1.
Then $C_a<\infty$ for all $a\in A$ if and only if the set $\{M\in [0,\infty)| M-a^ 2\in Q_L\}$ 
is  non empty  for all $a\in A$ and we have
$$C_{a^ 2}=(C_a)^ 2=\inf \{M\in [0,\infty)| M-a^ 2\in Q_L\}.$$
\end{lemma}
\begin{proof}We get as in the proof of Theorem \ref{PreMR} that  $C_{a^ 2}-a^ 2\in Q_L$. This yields $C_{a^ 2}\ge \inf \{M\in [0,\infty)| M-a^ 2\in Q_L\}.$

Now if $M-a^ 2\in Q_L$ then   
$$L(a^ {2n})\le ML(a^ {2(n-1)})\le \dots \le M^ nL(1)=M^ n,$$
 for all $n\in \mathbb N$, which yields $C_a\le \sqrt M$ and finishes the proof.

\end{proof}

The next theorem is a consequence of Theorem \ref{MR} and the previous lemma and is a slight modification of  \cite{DA2},Theorem 1.4.

Note that this shows that \cite{DA2},Theorem 1.4 and Theorem \ref{MR} are equivalent.
\begin{theorem}\label{FA}Let $A$ be an unital commutative $\mathbb R$-algebra and $Q\subseteq A$ with $1\in Q$. Let $(M_a)_{a\in A}$ a family of positive real numbers and 
$$K=\{\alpha \in X(A)||\alpha(a)|\le M_a, a\in A\,\, \text{and}\,\,\alpha (q)\ge 0,q\in Q\}$$
For a linear function $ L:A\to \mathbb R$ the following conditions are equivalent:
\begin{enumerate}
\item[(i)]
$M_a^ 2-a^ 2\in Q_L$ for every $a\in A$ and $Q\subseteq Q_L$;
\item[(ii)]There is a positive Radon measure $\mu$ on $K$ such that

$$L(a)=\int_{K}\alpha (a)d\mu(\alpha),a\in A$$
\end{enumerate}
\end{theorem}

\subsection{Functional positive on an archimedean quadratic module}

Next, we give a short proof of a classical result. See, for example, \cite {SCH4}, Corollary 12.47.
\begin{prop}Let $A$ be an unital commutative $\mathbb R$-algebra, and $L(1)=1$. If $L(Q)\subseteq [0,\infty)$ for some archimedean quadratic module $Q$ in $A$, then there exists a unique representing Radon measure $\nu$ for $L$ with support contained in
$$\{\alpha\in X(A):\alpha(a)\ge 0\,\,\text{for all} \,\,a\in Q\}$$
\end{prop}
\begin{proof}The fact that $Q$ is archimedean yields that
for every $a\in A$ there exists a positive number $M_a$ such that $M_a^2-a ^2\in Q$. Now we obtain as in the proof of Lemma \ref{minC_a}, that
 $C_a\le M_a$ and our result is a consequence of Theorem \ref{MR}.
\end{proof}

\subsection{Schm\text{\"u}dgen Theorem for compacts semi-algebraic sets}
\begin{theorem}
Let $\{R_1,\dots,R_m\}$ be a finite set of polynomials from 

$\mathbb R[X_1,\dots,X_d]$.  Suppose that the set 
$$K=\{x\in \mathbb R^d:R_j(x)\ge 0\quad\text{for}\quad j=1,\dots,m\}$$
is compact. Then a functional $L:\mathbb R[X_1,\dots,X_d]\to \mathbb R$ is a $K-$ moment functional if and only if $L$ and $P\to L(R_{j_1}\dots R_{j_k}P)$ are positive semidefinite for all possible choices $j_1,\dots,j_k$ of pairwise different numbers from $\{1,\dots,m\}$.
\end{theorem}
\begin{proof}
Because $K$ is compact there is a $\rho>0$ such that 
$K\subset B_\rho$ where $B_\rho=\left\{(x_1,\dots,x_d)\in \mathbb R^d:x_1^ 2+\dots+x_d^ 2\le \rho^2\right\}$ and $\rho^2-x_1^ 2+\dots-x_d^ 2>0$ for every $(x_1,\dots,x_d)\in K$.

According to \cite{BoCoR},Corollary 4.4.3,(ii)  there are polynomials $G$ and $H$ in the cone generated by the polynomials $P^2$ and the $P^2R_{j_1}\dots R_{j_k}$ for all possible choices $j_1,\dots,j_k$ of pairwise different numbers from $\{1,\dots,m\}$, such that
$$(\rho^2-X_1^2-\dots-X_d^2)G=1+H$$

Because $P\mapsto L(GP)$ and $P\mapsto L((\rho^2-X_1^2-\dots-X_d^2)GP)$ are positive semidefinite, it results from Proposition \ref{Msf}, that
there is a positive Radon measure $\nu$ on 
$$B_\rho=\left\{(x_1,\dots,x_d)\in \mathbb R^d:x_1^ 2+\dots+x_d^ 2\le \rho^2\right\}$$
such that
$$L(GX^\alpha)=\int_{B_\rho}x^\alpha d\nu(x),\alpha\in \mathbb N^d.$$
Consequently
$$L((\rho^2-X_1^2-\dots-X_m^2)GX^\alpha)=\int_{B_\rho}(\rho^2-x_1^2-\dots-x_d^2)x^\alpha d\nu(x),\alpha\in \mathbb N^d.$$
Hence  the function $P\mapsto L((\rho^2-X_1^2-\dots-X_m^2)GP)$ is a moment functional $P\mapsto M(P)$ on $B_\rho$.
We have $M=L+Q$ where $Q$ is the functional $P\mapsto L(HP).$ 
This yields
$$\sup_{n\in \mathbb N}\sqrt[2n] {L((X_1^ {2}+\dots+X_m^ {2})^ {2n})}\le
\sup_{n\in \mathbb N}\sqrt[2n] {M((X_1^ {2}+\dots+X_m^ {2})^ {2n})}\le\rho^2
$$
because $Q$ is a positive semidefinite functional and $M $  is a moment functional on $B_\rho.$
Consequently applying once again Proposition \ref{Msf} the function $L $  is a moment functional on $B_\rho.$

Because

$$K=\{x\in \mathbb R^d:R_j(x)\ge 0\quad\text{for}\quad j=1,\dots,m\}\subseteq B_\rho$$
it results from Theorem \ref{MR} that the support of the representing measure for $L$ is a subset of $K$.
\end{proof}
\begin{remark}A similar proof is sketched in \cite{DA3},Remark 2.4.
\end{remark}
\section{Intrinsic characterizations of the moment functions on an unital commutative semigroup with involution}
An unital semigroup with involution is a semigroup $(S,\cdot)$ with neutral element $1$ and a function $^*:S\to S$ satisfying
\begin{enumerate}
\item $(st)^* =t^*s^*,s,t\in S$;
\item $(s^*)^*=s, s\in S$.

\end{enumerate}

In this section, we consider unital commutative semigroups with involution.

\begin{df}Let $S$ be a semigroup. A function  $f:S\to \mathbb C$ is called \it{positive semidefinite} if 
$$\sum_{j,k=1}^nc_j\bar c_kf(s_j^*s_k)\ge 0$$
for all $n\in \mathbb N$ and every choice of $s_1,\dots,s_n\in S$ and $c_1,\dots c_n \in \mathbb C$.
\end{df}
\begin{df}Let $S$ be a semigroup. A function  $v:S\to [0,\infty)$ is called an \it{absolute value} if 
\begin{enumerate}
\item $v(1)=1$;
\item $v(st) \le v(s)v(t),s,t\in S$;
\item $v(s^*)=v(s), s\in S$.

\end{enumerate}
\end{df}

\begin{df}Let $S$ be a semigroup. A function  $\alpha :S\to \mathbb C$ is called a \it{character} if 
\begin{enumerate}
\item $\alpha(1)=1$;
\item $\alpha (st)=\alpha(s)\alpha (t),s,t\in S$;
\item $\alpha(s^*)=\overline{\alpha(s)}, s\in S$.

\end{enumerate}
\end{df}

The set of characters is denoted by $X(S)$, and we equip it with the topology of pointwise convergence.

The definitions for moment function and representing measure are the same as in the case of real algebra. 
\begin{theorem}\label{MR2}
Let $ f:S\to \mathbb C$ be a positive semidefinite function. Then, there exists a unique representing measure  $\nu_f$ for $f$ ,with compact support ,
 if and only if 
$$R_s=\sup_{n\in \mathbb N}\sqrt[2n] {f(s^n(s^*) ^{n})}<\infty,s\in S$$

Moreover, in this case, the support of the measure $\nu_f$ is a subset of the set

$$K_f=\{\alpha \in X(S):|\alpha(s)|\le R_s\text { for all } s\in S\}.$$

\end{theorem}
\begin {proof}Suppose $R_s\in \mathbb R$ for every $s\in S$.
Because from  Cauchy-Buniakovski-Schwarz inequality (see \cite{BCR}, 4.1.6) we get

$(f(s^n(s^*) ^{n}t^m(t^*) ^{m}))^ 2\le f(s^{2n}(s^*) ^{2n}) f(t^{2m}(t^*) ^{2m}), s,t\in S, m,n\in \mathbb N$,

the function $s\mapsto \sup_{n\in \mathbb N}\sqrt[2n] {f(s^ {n})s^ {*n}}=R_s$ is an absolute value on $S$ such that $|f(s)|\le R_s,s\in S$. Consequently, it results from \cite{BCR}, Theorem 2.5 that there exists a unique representing measure  $\nu_f$ for $f$ on $K_f$.

Now, the proof can be easily completed.
\end{proof}
Next, we give an application of the previous theorem.
Let $(\mathbb N_0^2,\cdot,*) $ be the semigroup  such that $(m,n)\cdot(p,q)=(m+p,n+q)$ and $(m,n)^*=(n,m)$. Note that $(m,n)\cdot(0,0)=(m,n)$.
 
It is easy to see that $X(\mathbb N_0^2)$ is $\mathbb C$ via the function $z\mapsto ( (m,n)\mapsto z^n \bar z ^m)$.

The following theorem is an immediate consequence of the previous theorem and relaxes the requirements from \cite {ATZ}, Theorem 2.1 and \cite {DA2}, Theorem 1.4.
\begin{theorem}
Let  $f:\mathbb N_0^2\to \mathbb C$ and $r>0$. The following conditions are equivalent 
\begin{enumerate}
\item The function $f$  is positive semidefinite and  there exists a number $C>0$ such that $f(n,n)\le Cr^{2n}, n\in \mathbb N$ ;
\item there is a Radon measure $\mu $ on $K=\{z\in\mathbb C;|z|^2\le r^2\}$ such that
$$f(m,n)=\int_Kz^n \bar z ^md\mu(z),(m,n)\in \mathbb N_0^2.$$
\end{enumerate}
\end{theorem}

\textbf{Acknowledgments.}

The author thanks Professor Konrad Schm\"udgen for providing information on the reference \cite{SCH5}  where is established a Positivstellensatz for  an Archimedean cone, which 
is, in general,neither a semiring nor a quadratic module.

\textbf{Data availability declaration}

I do not analyze or generate any datasets because my work proceeds within a theoretical and mathematical approach. One can obtain the relevant materials from the references below.
\textbf{Ethics declaration}

The author declares no competing interests.
\bibliographystyle{amsplain}
\bibliography{references}

\end{document}.

\textbf{Ethics declaration}

The author declares no competing interests.
Using the previous result, we can obtain the following solution of the moment problem on a generalized simplex.

\begin{prop}\label{Msimplex}
Let $A$ be an unital algebra with generators $a_1,\dots,a_m$.
For a linear function  $L:A\to\mathbb R $ such that $L(1)=1$ the following conditions are equivalent :
\begin{enumerate}

\item the functional $L$ is positive semidefinite, and we have  
$$d_{a_1}\ge p_1,\dots,d_{a_m}\ge p_1,D_{a_1+\dots+a_m}\le r;$$
\item there is a unique positive Radon measure $\mu$ on 
\begin{equation*}K=\left\{\alpha\in X(A):\alpha(a_1)\ge p_1,\dots\alpha(a_m)\ge p_1,\alpha(a_1)+\dots+\alpha(a_m)\le r\right\}\end{equation*}
such that
$$L(a)=\int_K\alpha(a)d\mu(\alpha), a\in A.$$
\end{enumerate}
\end{prop}